\documentclass[11pt,a4paper,twoside,draft]{amsart}

\usepackage[T1]{fontenc}
\usepackage{amssymb}

\newcommand{\beqla}[1] {\begin {eqnarray}\label{#1}}
\def \eeq {\end {eqnarray}}
\newcommand{\beqno}{\begin{eqnarray*}}
\newcommand{\eeqno}{\end{eqnarray*}}

\newtheorem{theorem}{Theorem}
\newtheorem{lemma}{Lemma}

\newtheorem{corollary}{Corollary}
\newtheorem{definition}{Definition}
\theoremstyle{definition}

\theoremstyle{remark}
\newtheorem{remark}[theorem]{Remark}

\newcommand{\sst}{\scriptstyle}

\newcommand{\half}{\frac{1}{2}}

\newcommand{\real}{{\mathbf R}}

\newcommand{\nanu}{{\mathbf N}}

\newcommand{\torus}{{\mathbf T}}

\newcommand{\refeq}[1]{(\ref{#1})}

\newcommand{\homog}{{\widetilde W}}

\newcommand{\smooth}{C_0^\infty}
\DeclareMathOperator{\cspan}{\overline{span}}
\DeclareMathOperator{\supp}{{\rm supp}}
\DeclareMathOperator{\expec}{{\mathbf E}}
\DeclareMathOperator{\prob}{{\mathbf P}}

\newcommand{\joille}{\, |\, }

\begin{document}

\title[Local independence of fractional Brownian motion]
{Local independence of fractional Brownian motion}

\author{Ilkka Norros}
\address{VTT Technical Research Centre of Finland, 
P.O.\ Box 1000, 02044 VTT, Finland}
\email{ilkka.norros@vtt.fi}

\author{Eero Saksman}
\address{University of Helsinki, Department of Mathematics and Statistics,
P.O.~Box~68 (Gustaf H\"allstr\"omin katu 2b), FIN-00014 University of Helsinki,
Finland}
\email{eero.saksman@helsinki.fi}


\date{\today}



\begin{abstract}
Let $ \sigma_{(t,t')}$ be the sigma-algebra generated by the
differences $X_s-X_{s'}$ with $s,s'\in (t,t')$, where $(X_t)_{-\infty
<t<\infty}$ is the fractional Brownian motion with Hurst index $H\in
(0,1).$ We prove that for any two distinct timepoints $t_1$ and $t_2$
the sigma-algebras $\sigma_{(t_1-\varepsilon,t_1+\varepsilon)}$ and
$\sigma_{(t_2-\varepsilon,t_2+\varepsilon)}$ are asymptotically independent as
$\varepsilon\searrow 0$. We show this in the strong sense that
Shannon's mutual information between the two $\sigma$-algebras tends
to zero as $\varepsilon\searrow 0$. Some generalizations and
quantitative estimates are provided also.
\end{abstract}

\maketitle

\section{Introduction}\label{se:intro}

Let $X=(X_t)_{-\infty <t<\infty}$ be the standard fractional Brownian
motion (FBM) with Hurst index $H\in (0,1)$. Thus, $X_t$ is a centered
Gaussian process with stationary increments and variance function
$\expec X_t^2=|t|^{2H}$ (see, e.g.,
\cite{Kolmogorov,Samorodnitsky}). The parameter value $H=\half$ yields
the standard Brownian motion. FBM is a $H$-self-similar process, that
is, $(X_{at}){\buildrel(d)\over=}(a^HX_t)$, where $\buildrel(d)\over=$
stands for the equality of finite-dimensional distributions.

When $H\not=\half$, the increments of $X$ on disjoint time intervals
are always dependent --- negatively correlated for $H<\half$ and
positively correlated for $H>\half$. Moreover, when $H>\half$, the
sequence $X_1$, $X_2-X_1$, $X_3-X_2,\ldots$ is long-range
dependent. i.e.\ $\sum_{i=1}^\infty\expec X_1(X_{i+1}-X_i)=\infty$
(see \cite{Cox,Samorodnitsky}). FBMs with $H>\half$ are often used in
applications as a mathematical model for far-reaching dependence. 
 
However, as we show in this paper, `small and distant' events in FBMs
are nevertheless asymptotically independent. This holds both as
asymptotic orthogonality of the Gaussian subspaces generated by the
processes $(X_t)_{|t|<1}$ and $(X_{n+t}-X_n)_{|t|<1}$ as $n\to\infty$,
and in the stronger sense that the mutual (Shannon) information
$I((X_t)_{|t|<1}:(X_{n+t}-X_n)_{|t|<1})$ is finite and decays to zero
as $n\to\infty$. By self-similarity, this is equivalent to considering
the increment processes around two fixed timepoints,
$(X_{s+u}-X_s)_{|u|<\varepsilon}$ and
$(X_{t+u}-X_t)_{|u|<\varepsilon}$, as $\varepsilon\searrow0$. We
propose to call this latter property {\em local independence}.  

Our
paper was motivated by \cite{Mandjes}, where FBM's local independence
property was needed, but attempts to find this result from literature
were unsuccessful. Very recently, however, J.\ Picard \cite{Picard}
has proven the asymptotic orthogonality result using a different
technique.  The more functional analytic approach of the present note has the advantage of giving
very precise estimates both for the rate of asymptotic orthogonality,
  and for the
much stronger property of  asymptotically vanishing mutual information.

The structure of the paper is as follows: in the first section we
briefly recall certain facts about Sobolev spaces with fractional
smoothness index -- these spaces are the main tool in our approach.  We have
tried to make the exposition readable for the readers with no previous
knowledge on these spaces. The second section reviews the basic facts
on the Gelfand-Yaglom theory of mutual information between Gaussian
spaces. The third section contains the proof of our main results. The
results are obtained in a quantitative form in terms of the relative
size of the time intervals involved.  Finally, the fourth section
briefly considers the higher dimensional case and states open
questions.

\section{Preliminaries I: the fractional Sobolev spaces}\label{se:sobolev}

We shall apply the common notation for uninteresting constants. They
will all be denoted by the letter $c$, and its value can vary inside a
single estimate. The notation $a\sim b$ means that the ratio of the
(positive) quantities $a$ and $b$ stays bounded from below and above
as the parameters of interest vary. The inner product of elements
$\phi$ and $\psi$ of a Hilbert space $\mathcal{H}$ will be denoted as
$(\phi,\psi)_{\mathcal{H}}$, and the angle $\sphericalangle (A,B)$
between subspaces $A$ and $B$ of ${\mathcal{H}}$ is defined by
$$
\cos (\sphericalangle (A,B)):=
\sup\left\{\frac{(U,V)_{\mathcal{H}}}{\|U\|_{\mathcal{H}}
\|V_{\mathcal{H}}\|}:U\in A,V\in B\right\}.
$$

Suitable references for this section are e.g. \cite[Section 6]{Rudin}
or selected parts of \cite{Stein1}.  Vastly more information can be
found in Triebel's monographs, like \cite{Triebel}.  Actually only
very little of the theory of Sobolev spaces is needed, and we try to be
as self-contained as possible.

The Fourier transform of a tempered distribution $f$ on $\real^n$ is
defined as 
$$
\widehat f(\xi ):=(2\pi )^{-n/2}\int_{\real^n }e^{-ix\cdot
\xi}f(x)\, dx.
$$
We shall employ the notation $\langle \lambda, \mu \rangle$
for the distributional pairing, assuming that it is well-defined for
$\lambda$ and $\mu $. Recall that the convolution $\lambda*\phi$ is
always defined if $\lambda$ is a Schwartz distribution and
$\phi\in\smooth (\real^n)$, and its Fourier transform is the product
$(2\pi )^{n/2}\widehat\phi \, \widehat\lambda .$ Moreover, by the
definition of the Fourier transform, the Parseval identity can be
written in the form
$$
\langle \lambda, \overline{\phi}\rangle =
\langle \widehat\lambda, \overline{\widehat\phi}\rangle.
$$

Let $s\in\real$. The Sobolev space $W^{s,2}(\real^n)$ is defined as
the Hilbert space of tempered distributions $f$ on $\real^n$ such that
the Fourier transform $\widehat f(\xi )$ is a locally integrable
function with the property \beqla{eq1.1} \| f\|_{s,2}:= \|
f\|_{W^{s,2}}:=\left( \int_{\real^n} | \widehat f( \xi )|^2(1+|\xi
|^2)^s \right)^{1/2} <\infty .  \eeq Our normalization constant for
the Fourier transform makes sure that $W^{0,2}(\real^n)= L^2(\real^n)$
isometrically.

In the distributional pairing, the
isometric dual of $W^{s,2}(\real^n)$ is $W^{-s,2}(\real^n)$.  Moreover, the norm
increases as $s$ increases, and for integers $k\in\nanu$ we have that
\beqla{eq1.3} \| f\|^2_{k,2}\sim \int_\real
(|f(x)|^2+\sum_{|\alpha|=k}|f^{(\alpha )}(x)|^2)\, dx.  
\eeq 
Obviously all these spaces are translation invariant, and one may verify that
multiplication by an element in $C_0^\infty (\real^n ) $ is
continuous.

We next recall the homogeneous Sobolev spaces $\homog^{s,2}(\real^n
)$.The norm is replaced by 
\beqla{eq1.6} 
\| f\|_{\homog^{s,2}(\real^n
)}:=\left( \int_{\real^n} | \widehat f( \xi )|^2| \xi |^{2s}\,d\xi
\right)^{1/2} <\infty .  
\eeq 
This norm is certainly well-defined at
least for all $f\in \smooth (\real^n)$, although even then it may take
the value $\infty$ if $s<-n/2.$ In defining the Hilbert space
$\homog^{s,2}(\real^n )$ there indeed arises some complications in the
definition, due to the fact that the $| \xi|^{2s}$ can be either 'too
big' or 'too small' near origin. However, for our main result it is
enough to consider the case $n=1$ and $|s| < 1/2$, and then these
difficulties disappear. For these values of the parameters the
homogeneous spaces are simpy defined as the (inverse) Fourier transform of the
weighted space $L^2_\mu (\real )$, where the weight is of the form
$\mu (d\xi)=|\xi |^{2s}.$ By Cauchy-Schwartz any function in this
weighted space is a locally integrable function, and thus defines a
distribution in a natural way.  On the other hand, every Schwartz test
function belongs to this weighted space, which can be used to show
that $\smooth (\real )\subset \homog^{s,2}$ is a dense
subset. Moreover, the isometric duality
$(\homog^{s,2}(\real ))^*=\homog^{-s,2}(\real )$ 
holds via the distributional duality
$$
\langle \phi,\psi\rangle =\int_\real \phi (x)\psi (x)\, dx.
$$
The pairing is originally defined only for test functions, but it 
extends to elements
$\phi \in\homog^{s,2} (\real) $  and $\psi \in\homog^{-s,2} (\real) $ by
continuity and density.

We then fix $s\in (-1/2,1/2)$ together with an open interval $I\subset\real$ ($I$ can well be
unbounded) and define the Sobolev-functions over this interval.  First
of all we denote by $\homog_0^{s,2}(I)$ the closure of $C_0^\infty
(I)$ in the space $\homog^{s,2}(\real ).$ Clearly all the elements in
$\homog_0^{s,2}(I)$ are distributions supported on $\overline{I}.$ We
will also need the space $\homog^{s,2}(I)$ which consists of
restrictions of elements of $\homog^{s,2}(\real)$ on the interval
$I$. Thus $\homog^{s,2}(I)=\{ g_{|I} : g\in \homog^{s,2}(\real )
\}$. This space is naturally normed by the induced quotient norm
$$
\| f\|_{ \homog^{s,2}(I) } := \inf \{ \| g\|_{\homog^{s,2}(\real )} :
g_{|I} =f \}.
$$
In a similar vain one defines the non-homogeneous space $W^{s,2}(I)$ by setting 
$W^{s,2}(I)=\{ g_{|I} : g\in W^{s,2}(\real ) \}$ and introducing the
quotient norm
$$
\| f\|_{ W^{s,2}(I) } := \inf \{ \| g\|_{W^{s,2}(\real )}  : g_{|I} =f \}.
$$
This definition makes sense for all $s\in\real .$
One may easily verify that $ \| f\|^2_{W^{1,2}(I)}\sim
\int_I(f'^2(x)+f^2(x))\, dx , $ where $f'$ is the distributional
derivative of $f.$

Since $\homog_0^{s,2}(I)\subset\homog^{s,2}(\real) $ is a (closed) subspace,
we deduce by standard Hilbert space theory that isometrically
\beqla{eq:1.0}
(\homog_0^{s,2}(I))'=\homog^{-s,2}(I)\quad \mbox{and }
\quad (\homog^{s,2}(I))'=\homog_0^{-s,2}(I)
\eeq
through the pairing $\langle \phi,\psi\rangle =\int_I\phi (x)\psi
(x)\, dx$ (extended again by continuity). There is thus a natural
isometry $G: \homog^{-s,2}(I)\to \homog_0^{s,2}(I)$ in such a way that
\beqla{eq:1.00}
( \phi,G\psi )_{\homog_0^{s,2}(I)}=\int_I \phi (x)\psi (x)\, dx
\eeq
for smooth elements $\phi$ and $\psi$. Again this extends for any
$\phi\in \homog_0^{s,2}(I)$ and $\psi\in \homog^{-s,2}(I)$ by
continuity.

In the Lemma below the assumption $|s|<\half$ is crucial. 

\begin{lemma}\label{le:1.10}
Let $s\in (-\half,\half)$ and let $I\subset\real$ be an open interval
of length 1.

\smallskip

\noindent {\rm (i)}\quad Multiplication by the signum function extends
to a bounded linear operator on $\homog^{s,2}(\real )$.  In other words, $\|
\chi_{(-\infty ,0)}f\|_{\homog^{s,2}(\real )}, \|
\chi_{(0,\infty)}f\|_{\homog^{s,2}(\real )} \leq c\| f\|_{\homog^{s,2}(\real )}$ for
all $f\in\smooth (\real )$. The same statement remains true if  $\homog^{s,2}(\real )$ is
replaced by  $W^{s,2}(\real )$.

\smallskip

\noindent {\rm (ii)}\quad $\homog_0^{s,2}(I)=\{f\in \homog^{s,2}(\real ): \supp (f)\subset\overline{I}\} .$

\smallskip

\noindent {\rm (iii)}\quad We have
$\homog_0^{s,2}(I)=\homog^{s,2}(I)=W^{s,2}(I)$ with equivalent norms
(the constant of isomorphism does not depend on the location of the
interval $I$).

\smallskip

\noindent {\rm (iv)}\quad There is a continuous inclusion
$W^{1,2}(I)\subset \homog_0^{s,2}(I)$, and this natural imbedding is a
Hilbert-Schmidt operator.
\end{lemma}
\begin{proof}
(i)\quad The statement is well-known, see \cite[First Lemma in Section
2.10.2.]{Triebel}. Actually, up to a constant the multiplication by
the signum function corresponds to the action of the Hilbert transfrom
on the Fourier side. Hence the claim follows from the fact that
$|\xi|^{2s}$ is a Muckenhoupt $A^2$-weight on $\real$ for any $s\in
(-1/2,1/2)$, see \cite[Corollary, V.4.2, V.6.6.4]{Stein2}. In a similar way,
by checking that $(1+|\xi|^{2})^{s}$ is a Muckenhoupt weight one obtains
 the statement concerning $W^{s,2}$.

(ii)\quad Let $f\in \homog^{s,2}(\real )$ with $\supp (f)\subset \overline{I}.$
We will show that one may approximate $f$ in norm by the elements of $\smooth (I).$
The dilation $\lambda\to f(\lambda\cdot )$ is a continuous map from a neighbourhood
of $1$ into $\homog^{s,2}(\real )$. Hence, by approximating $f$ with a suitable
dilation we may assume that $\supp (f)$ is contained in $I$. Finally, we
then obtain the required approximant  by a standard  mollification.

(iii)\quad  By the translation invariance of the spaces, the independence on the location
of the interval $I$ is obvious. The first equality is an easy consequence of parts (i) and (ii).
Towards the second equality,
let us first verify that $\smooth (I)$ is dense in $W^{s,2}(I)$. 
By part (i), if $f\in W^{s,2}(I)$ then also $\chi_If\in W^{s,2}(\real ),$
where $\chi_If$ stands for the zero continuation of $f$ to $\real .$
Exactly as in part (ii) we show by dilation and convolution approximation that
$\chi_If$ is in the closure of $\smooth (I)$
in $W^{s,2}(\real )$, which clearly yields the claim.

Hence it remains to show that 
\beqla{eq2:1000}
\|f\|_{W^{s,2}(\real )}\sim\| f\|_{\homog^{s,2}(\real )}\quad \mbox{for}\:\: f\in \smooth (I).
\eeq 
We may clearly assume that $I=(0,1)$. Let us first consider the inequality 
\beqla{eq1.5} 
\|f\|_{W^{s,2}(\real )}\leq c\| f\|_{\homog^{s,2}(\real )}.  
\eeq 
This is immediate if
$s\leq 0$. If $s\in (0,1/2)$ we choose a cut-off function $\phi\in
\smooth(-1,2)$ such that $\phi=1$ on the interval $[-1/2,3/2]$. Let us
decompose
$$
f= \phi f_1+\phi f_2,
$$
where $\widehat f_1= \chi_{[-1,1]}\widehat f,$ and $\widehat f_2=
\widehat f-\widehat f_1.$ Then obviously $\|\phi f_2\|_{W^{s,2}(\real )}\leq
c\| f_2\|_{W^{s,2}(\real )}\leq c\| f\|_{\homog^{s,2}(\real )}$.  Moreover,
$$
f_1(x)=\frac{1}{\sqrt{2\pi}}\int_{-1}^1e^{ix\xi}\widehat f(\xi )\,d\xi,\quad
f'_1(x)=\frac{i}{\sqrt{2\pi}}\int_{-1}^1e^{ix\xi}\xi \widehat f(\xi )\,d\xi\quad ,
$$ 
where, by Cauchy-Schwarz, $\int_{-1}^1|\widehat f(\xi )| d\xi \leq c\|
f\|_{\homog^{s, 2}(\real )}.$ Hence $\| f_1\|_\infty +\| f'_1\|_\infty\leq c\|
f\|_{\homog^{s, 2}(\real )}$ and we obtain that $ \| \phi
f_1\|_{W^{s,2}(\real )}
\leq \| \phi f_1\|_{W^{1,2}(\real )}\leq
c\| f_1\|_{\homog^{s,2}(\real )}.$ By combining these estimates \refeq{eq1.5}
follows.

In turn, the converse inequality 
\beqla{eq1.51} 
\|f\|_{\homog^{s,2}(\real )}\leq c \| f\|_{W^{s,2}(\real )}.  
\eeq 
is immediate
if $s\geq 0$. It clearly follows for negative $s\in (-1/2,0)$ if we
verify that in our situation $\| \widehat f\|_{L^\infty(-1,1)}\leq c
\| f\|_{W^{s,2}(\real )}.$ This is seen by observing that
$$
\widehat f(\xi )=  \frac{1}{\sqrt{2\pi}}\langle f(x ), \phi (x) e^{-i\xi x}\rangle ,
$$
where $\sup_{-1\leq \xi \leq 1}\|\phi (x) e^{-i\xi x}\|_{W^{-s,2}(\real )}\leq c.$ 

(iv)\quad By part (iii), the claim is a consequence of the well-known
Hilbert-Schmidt property of the inclusion $W^{1,2}(I)\subset
W^{s,2}(I)$. Since we have not been able
to find a convenient reference, the simple proof is sketched  here. We may
assume that $I=(-1/2,1/2)$ so that $I\subset (-\pi ,\pi ]=:\torus ,$
where $\torus$ stands for the 1-dimensional torus. By applying a simple
extension one may consider the
spaces in question as closed subspaces of the corresponding Sobolev spaces
$H^1(\torus )$ and $H^s(\torus )$
on the torus,
where  for $f=\sum_{n=-\infty}^\infty a_ne^{inx}$ and $u\in\real$ one sets
$\|f\|_{H^u(\torus )}^2=\sum_{n=-\infty}^\infty  (1+|n|)^{2u}|a_n|^2$
(see e.g.\ \cite{Rudin}). By considering the natural
 orthogonal  basis $((1+|n|)^{-u} e^{inx})_{n=-\infty}^\infty$ we
see that the embedding $H^1(\torus )\subset H^s(\torus )$ is equivalent to
the diagonal operator
with the diagonal elements $((1+|n|)^{s-1} )_{n=-\infty}^\infty .$ This is
Hilbert-Schmidt as
$\sum_{n=-\infty}^\infty  (1+|n|)^{2s-2}<\infty $.

\end{proof}

We shall need the formula for the Fourier transform of the function
$u_\alpha (x):=|x|^{-\alpha},$ where $x\in\real^n$ and $\alpha\in
(0,n).$ It is well-known, see e.g. \cite[V 1. Lemma 2, p.117]{Stein1}, that
\beqla{eq:1.90} 
\widehat u_\alpha (\xi )=d_{n,\alpha}|\xi
|^{\alpha-n},\quad\mbox{where}\;\; d_{n,\alpha}:=
2^{n/2-\alpha}\frac{\Gamma ((n-\alpha )/2)}{\Gamma (\alpha /2)}.  
\eeq

\begin{lemma}\label{le:1.2} Assume that $s\in (-1/2,1/2).$

\noindent {\rm (i)}\quad Let $\alpha >0$, $\alpha\not= 1$ and denote
$f_{\alpha}(x)=(1+|x|)^{-\alpha}.$ Then $f_{\alpha}\in \homog^s(\real
)$ for $\alpha >1/2-s$.

\smallskip

\noindent {\rm (ii)}\quad Let $\alpha >1/2+s,$ $\alpha\not= 1$. Then
for any $k>0$ there is a constant $c (\alpha ,s) >0$ such that
$\|(k-\cdot )^{-\alpha}\|_{\homog^{-s,2}((-\infty ,0))}=c(\alpha ,s)k^{1/2+s-\alpha}.$
In other words, 
\beqla{eq:1.101}
\sup_{\scriptstyle
\begin{array}{l} \sst\|\phi \|_{\homog_0^{s,2}( (-\infty
,0))}\leq 1
\end{array}
}
\int_{-\infty}^0 (k-x)^{-\alpha}\phi (x)\,
dx\;\;=c(\alpha ,s)k^{1/2+s-\alpha}.  
\eeq
\end{lemma}

\begin{proof}
(i)\quad Choose a smooth cut-off function $\phi\in \smooth (\real )$
such that $\phi =1$ in a neighbourhood of the origin. Compose
$$
f_\alpha (x)= \phi (x)f_\alpha (x)+(1-\phi (x))(f_\alpha
(x)-|x|^{-\alpha})+ (1-\phi (x))|x|^{-\alpha} =: g_1(x)+g_2(x)+g_3(x).
$$
Obviously $g_1\in L^1(\real )\cap W^{1,2}(\real )\subset
\homog^{s,2}(\real )$ for all $|s| <1/2.$ An easy eastimate shows that
the same holds for $g_2.$ Moreover, we observe that $(d/dx)g_3\in
L^2(\real )$. Hence $\int_{\real}| \xi |^2|\widehat g_3(\xi )|^2
<1$. Thus the inclusion $g_3\in \homog^{s,2}(\real )$ holds if and
only if the integral $\int_{-1}^1|\xi|^{2s}|\widehat g_3(\xi
)|^2\,d\xi$ is finite.

Consider first the case $\alpha >1.$ Then $g_3\in L^1(\real )$, so
that $\widehat g_3$ is bounded and $g_3\in \homog^{s,2}(\real )$ for
all $|s| <1/2.$ Assume then that $\alpha\in (0,1).$ Then
$g_3(x)-|x|^{-\alpha}\in L^1(\real ),$ so that \refeq{eq:1.90} yields
$|\widehat g_3(\xi )-d_{1,\alpha} |\xi |^{\alpha -1}|\leq C.$ Thus
$\int_{-1}^1|\xi|^{2s}|\widehat g_3(\xi )|^2\,d\xi<\infty $ exactly
for $s>1/2 -\alpha .$

(ii)\quad The definition of the homogeneous Sobolev norm yields the
scaling rule $\| \phi(k\cdot )\|_{\homog^{s,2}(\real ) }= k^{s-1/2}\|
\phi(\cdot )\|_{\homog^{s,2}(\real )}.$ By using this fact, duality, and a substitution
$x=ky$ in the integral we are reduced to showing that
$$
\sup_{\scriptstyle\left\{\begin{array}{l} \sst\phi\in \smooth (-\infty
,0)\\ \sst\|\phi\|_{\homog^s(\real )}\leq
1\end{array}\right.}\int_{-\infty}^0 (1+|x|)^{-\alpha}\phi (x)\, dx
<\infty.
$$
By duality this follows immediately from the fact that
$(1+|x|)^{-\alpha}\in \homog^{-s,2}(\real )$ according to part (i) of
the Lemma.
\end{proof}

We finally remark that all the results stated in this section remain valid with identical proofs
for the Sobolev spaces that contain only real-valued functions.

\section{Preliminaries II: mutual information between Gaussian subspaces}\label{se:information}

In this section we present the needed facts from the Gelfand-Yaglom
theory of mutual information between Gaussian subspaces. In order to
recall the general concept of mutual information, let
$(\Omega,\mathcal{F},P)$ be a probability space, and let $\mathcal{A}$
and $\mathcal{B}$ be sub-$\sigma$-algebras of $\mathcal{F}$. The
mutual (Shannon) information between $\mathcal{A}$ and $\mathcal{B}$
is defined as \cite{Kolmogorov2}
$$
I(\mathcal{A}:\mathcal{B})
:=\sup_{\{ A_j\}\{ B_k\}}\sum_{k,j}\prob (A_j\cap B_k)
  \log\left( \frac{\prob (A_j\cap B_k)}{\prob (A_j)\prob ( B_k)}\right).
$$
Here the supremum is taken over all $\mathcal{A}$-measurable
partitions $\Omega =\bigcup_{j=1}^n A_k$ and $\mathcal{B}$-measurable
partitions $\Omega =\bigcup_{k=1}^m B_k$ of the probability space
($n,m\geq 1$, $\prob (A_j)>0$ and $\prob (B_k) >0$ for all
$j,k$). 

For random variables $X:\Omega\to E$ and $Y:\Omega\to F$, where $E,F$
are measurable spaces, we set $I(X:Y):= I(\sigma (X):\sigma (Y)).$ 
Let $\mu_X$ (resp. $\mu_Y$, $\mu_{(X,Y)}$) be the
distribution (measure) of $X$ (resp. $Y$, $(X,Y)$) in the space $E$
(resp. $F$, $E\times F$). Then, one may check that $I(X:Y)=\infty$ if the measure
$\mu_{(X,Y)}$ is not absolutely continuous with respect to the product
measure $\mu_X\otimes \mu_Y$. Moreover, in the case where
$\mu_{(X,Y)}<<\mu_X\otimes \mu_Y $ we denote $p=
\frac{d\mu_{(X,Y)}}{d(\mu_X\otimes \mu_Y)}$ and have the formula
\beqla{eq-shannon-comp} 
I(X:Y)=\int_{X\times Y}\log(p) \, d(\mu_X\otimes\mu_Y).
\eeq

The Kullback-Leibler information characterizes the shift from
a probability measure $\mu$ to another probability measure $\nu$ on
the same measurable space, and it is defined as
$$
I_{KL}(\mu:\nu)
=\left\{
\begin{array}{ll}
\int\log\frac{d\mu}{d\nu}d\nu,&\quad\mbox{if }\mu<<\nu,\\
\infty,&\quad \mbox{otherwise}.
\end{array}
\right.
$$
Shannon's mutual information can be expressed in terms of the
Kullback-Leibler information as
\beqla{eq-shannon-kullback} 
I(\mathcal{A}:\mathcal{B})=I_{KL}
(P_{(\mathcal{A},\mathcal{B})}:P_{\mathcal{A}}\otimes P_{\mathcal{B}}),
\eeq 
where $P_{(\mathcal{A},\mathcal{B})}$ denotes the unique probability measure on
$(\Omega\times\Omega,\mathcal{A}\times\mathcal{B})$ satisfying
$P_{(\mathcal{A},\mathcal{B})}(A\times B)=P(A\cap B)$ for $A\in\mathcal{A}$,
$B\in\mathcal{B}$. 
Actually, this is obtained from \refeq{eq-shannon-comp}  by letting $X$ (resp. $Y$)
be the identity map $(\Omega ,\mathcal{F})\to (\Omega ,\mathcal{A})$ 
(resp. the identity map $(\Omega ,\mathcal{F})\to (\Omega ,\mathcal{B})$).

The following properties of mutual information are most conveniently
proven through the relation (\ref{eq-shannon-kullback}).

\begin{theorem}\label{shannonbasic}
\noindent{\rm (i)} $I(\mathcal{A}:\mathcal{B})\geq 0$ and equality holds
if and only if $\mathcal{A}$ and $\mathcal{B}$ are independent.

\noindent{\rm (ii)} $I(\mathcal{A}:\mathcal{B})$ is non-decreasing with
respect to $\mathcal{A}$ and $\mathcal{B}$. 

\noindent{\rm (iii)} If $\mathcal{A}_n\uparrow\mathcal{A}$ and
$\mathcal{B}_n\uparrow\mathcal{B}$, then 
$I(\mathcal{A}_n:\mathcal{B}_n)\uparrow I(\mathcal{A}:\mathcal{B})$.

\noindent{\rm (iv)} If $\mathcal{A}_n\downarrow\mathcal{A}$ and
$\mathcal{B}_n\downarrow\mathcal{B}$, and if 
$I(\mathcal{A}_n:\mathcal{B}_n)<\infty$ for some $n$, then 
$I(\mathcal{A}_n:\mathcal{B}_n)\downarrow I(\mathcal{A}:\mathcal{B})$.

\end{theorem}

When $X$ and $Y$ are finite-dimensional random vectors such that
$(X,Y)$ is a non-degenerate and centered multivariate Gaussian, one may
easily compute by using (\ref{eq-shannon-comp}) that
$$
I(X :Y)
=\frac12\log\frac{\det(\Gamma_X)\det(\Gamma_Y)}{\det(\Gamma_{(X,Y)})},
$$
where $\Gamma_Z$ denotes the covariance matrix of a Gaussian vector
$Z$. In particular, the information between random variables $X$ and
$Y$ with bivariate centered Gaussian distribution is
\begin{equation}
\label{bivariatenormalinfo}
I(\sigma(X):\sigma(Y))=-\log\sin\sphericalangle(X,Y).
\end{equation}

The theory of Shannon information between Gaussian processes was
developed by Gel'fand and Yaglom \cite{Gelfand}. Their fundamental
discovery was that one may express the information between two closed
subspaces $A$ and $B$ of a Gaussian space $\mathcal{G}$ in terms of
the spectral properties of the operator $T:= P_AP_BP_A$, where $P_A$
and $P_B$ stand for the orthogonal projections on $A$ and $B$,
respectively.  In order to explain their result, and for later
purposes, we first recall some basic notions of operator theory.

Let $S:E\to F$ be a bounded linear operator between the separable Hilbert
spaces $E$ and $F$. Let $\{e_i\}_{i\in I}$ be an orthonormal basis for $E$.
The Hilbert-Schmidt norm of $S$ is defined as
$$
\| S\|_{HS(E,F)}:=\big(\sum_{i\in I}\|Se_i\|^2_F\big)^{1/2}.
$$ 
This definition does not depend on the particular orthonormal basis
used.  In case $\| S\|_{HS}<\infty$ we say that $S$ is a
Hilbert-Schmidt operator.  Also it is clear that if $E$ (resp. $F$) is
a Hilbert subspace of a larger space $\widetilde E$ (resp. $\widetilde
F$), then $\|SP_E\|_{HS(\widetilde E,\widetilde F)}=\|S\|_{HS(
E,F)}$. In this sense it is not important to keep exact track on the
domain of definition and image spaces, and one usually abbreviates
$\|S\|_{HS(E,F)}=\|S\|_{HS}.$ For products of bounded linear operators
between (perhaps different) Hilbert spaces we have
\beqla{eq:4.80}
\| TS\|_{HS} \leq\| T\|_{HS} \| S\| \quad\mbox{and}\quad
\| ST\|_{HS} \leq\| S\|\| T\|_{HS}  .
\eeq

Let us then assume, in addition, that $S:E\to E$ is self-adjoint and
positive semi-definite, $S^*=S$ and $S\geq 0$. Then one may always define the
trace of $S$ by setting
$$
{\rm tr\,}(S):=\sum_{i\in I}( e_i,Se_i )
$$
Thus, ${\rm tr\,}(S)\in [0,\infty ]$. In the case that ${\rm tr\,}(S)
<\infty $ we say that $S$ is of trace class. Every trace class
operator $S$ is compact, and since we also assume $S\geq 0$, it has a decreasing sequence of positive eigenvalues
$\lambda_1\geq \lambda_2\geq\ldots \geq 0$, where each eigenvalue is
counted according to its multiplicity. It follows that
\beqla{eq:4.81}
{\rm tr\,}(S)=\sum_{\lambda_k>0}\lambda_k.
\eeq
We finally observe that if $S:E\to F$ is any bounded linear operator,
then $S^*S\geq 0$ is self-adjoint, and we may compute 
\beqla{eq:4.83}
{\rm tr\,}(S^*S)=\| S\|^2_{HS}.
\eeq

Let us then go back to the situation where $A,B$ are closed subspaces
of a Gaussian Hilbert space $\mathcal{G}$ and state the result of
Gelfand and Yaglom. Again $P_A$ and $P_B$ stand for the orthogonal
projections to the subspaces $A$ and $B$, respectively, and
$I(A:B):=I(\sigma\{X: X\in A\}: \sigma\{Y:Y\in B\}).$

\begin{theorem}\label{th:4.10}{\rm \cite{Gelfand}} 
Denote $T:=P_AP_BP_A$. The mutual information $I(A:B)$ is finite if and only
if $\| T\| <1$ {\rm (}i.e.\ $\sphericalangle (A,B)>0${\rm )} and the
operator $T$ is of trace class. Moreover, in this case 
\beqla{eq:4.84}
I(A:B)=\frac{1}{2}\sum_{k:\lambda_k >0}\log (\frac{1}{1-\lambda_k}),
\eeq 
where $\lambda_1\geq \lambda_2\geq \ldots$ are the eigenvalues of
$T$ in the decreasing order repeated according to their multiplicities.
\end{theorem}

A nice sketch of the derivation of the formula (\ref{eq:4.84}) is
included in a form of exercises in \cite[pp. 68--69]{McKean}. Assume
that $T$ is of trace class, and let $Z_1,Z_2,\ldots$ be an orthonormal
basis of $T\mathcal{G}$ consisting of eigenvectors corresponding to
the non-zero eigenvalues $\lambda_1\geq \lambda_2\geq \ldots$. It is not
difficult to see that $\{P_BZ_i\}$ is an orthogonal basis of
$P_BP_AP_B\mathcal{G}$, and, moreover, these bases are {\em mutually}
orthogonal: $(Z_i,P_BZ_j)=0$ for $i\not=j$. Since orthogonality
implies independence in the case of Gaussian random variables, it
follows that the information between $\sigma(A)$ and $\sigma(B)$ can
be expressed as the sum of the informations within the pairs
$(Z_i,P_BZ_i)$, given in (\ref{bivariatenormalinfo}):
\begin{eqnarray*}
I(A:B)&=&-\sum_i\log\sin\sphericalangle(Z_i,P_BZ_i)\\
&=&-\frac12\sum_i\log(1-\cos^2\sphericalangle(Z_i,P_BZ_i))
=-\frac12\sum_i\log(1-\lambda_i).
\end{eqnarray*}
Note that since
$\sphericalangle(A,B)=\inf_i\sphericalangle(Z_i,P_BZ_i)$, the
information between subspaces can be infinite even when they have a
positive angle. 

By invoking the Taylor
series of $x\mapsto\log(1/(1-x))$ we obtain for $x\in [0,1)$ that
\beqla{eq:4.85}
x \leq \log (\frac{1}{1-x})=\sum_{k=1}^\infty
 \frac{1}{k}x^k\leq x+\frac{1}{2}x^2(\frac{1}{1-x})\leq 
x(1+\frac{x}{2(1-x)}).
\eeq
Observe also that $T=(P_BP_A)^*(P_BP_A)$ and $\|T\|=\| P_BP_A\|^ 2.$
Moreover, $\lambda_1\leq\|P_BP_A\|\leq\|P_BP_A\|_{HS}.$ By combining
these observations and the facts \refeq{eq:4.81}--\refeq{eq:4.85} we
obtain a formulation suitable for our purposes:

\begin{corollary}\label{co:4.1} 
The angle between the spaces $A$ and $B$ satisfies $\cos
(\sphericalangle (A,B))=\|P_BP_A\|$. We have $I(A:B)<\infty$ if and
only if $\| P_BP_A\| <1$ and $\|P_BP_A\|_{HS}<\infty$. Moreover, in
this case
\beqla{eq:4.86}
\frac{1}{2}\|P_BP_A\|_{HS}^2
&\leq& I(A:B)\;\leq \;\frac{1}{2}\|P_BP_A\|_{HS}^2
\left( 1+ \frac{\|P_BP_A\|}{2(1-\|P_BP_A\|)}\right)
\eeq
\end{corollary}

Observe that the above estimate is asymptotically precise in the limit
$\|P_BP_A\|\to 0$, or, equivalently, as $\sphericalangle (A,B)\to
\pi/2$. Especially this is true in the limit $I(A:B)\to 0.$

\section{Statement and proof of the main results}\label{se:orthogonal}

In this section we consider the asymptotic independence of the local
spaces of FBMs. To be more exact, let us first define for any set
$S\subset\real$
$$
E_S:=\cspan\{
X_u-X_v:u,v\in S\},
$$
and the shorthand notation
$$
E_{t,\varepsilon}:=E_{(t-\varepsilon,t+\varepsilon)}.
$$
We consider the following two notions of local independence.

\begin{definition}
We say that the stochastic process $X$ possesses {\em local
independence in the weak sense}, if for any distinct $t_1$, $t_2$
$$
\sphericalangle(E_{t_1,\varepsilon},E_{t_2,\varepsilon})\to\frac{\pi}{2}
\quad\mbox{as }\varepsilon\searrow0.
$$
We say that the stochastic process $X$ possesses {\em local
independence (in the strong sense)}, if for any distinct $t_1$, $t_2$
$$
I(E_{t_1,\varepsilon}:E_{t_2,\varepsilon})\to0
\quad\mbox{as }\varepsilon\searrow0.
$$
\end{definition}

The term `weak' corresponds to its use in `stationarity in the weak
sense'. 

We will consider integrals of the form $\int_\real X_t\phi (t)dt$ for
smooth and compactly supported functions $\phi .$ The definition of
the integral poses no problems since $t\mapsto X_t$ is continuous with
respect to $L^2$-norm of random variables, whence it can be
e.g. defined as the limit of the corresponding Riemann sums (or as a
Bochner integral). Let us start with two simple lemmata.

\begin{lemma}\label{le:2.1}
For any $T\in\real$ and $a>0$ the elements
\beqla{eq2.10}
\int_\real \phi'(t)X_t\, dt,\quad \phi\in \smooth (T,T+a)
\eeq
are dense in $E_{(T,T+a)}.$
\end{lemma}

\begin{proof} 
By observing that $\int_\real \phi'(t)X_t\, dt= \int_\real
\phi'(t)(X_t-X_{T+a/2})\, dt$ we see that the elements in question are
contained in $E_{T,a}.$ Conversely, let $\phi\in\smooth (\real )$
satisfy $\int_\real \phi (t)\, =1 .$ Denote $\phi_\varepsilon
(x)=\varepsilon^{-1}\phi (x\varepsilon).$ By the $L^2$-continuity we
have that for any $t_1,t_2\in (T,T+a)$
$$
X_{t_1}-X_{t_2}=\lim_{\varepsilon\to 0}(\int_\real X_u(\phi_\varepsilon (t_1+u)
-\phi_\varepsilon (t_2+u))\, du.
$$ 
Observe that we may write $\psi'=\phi_\varepsilon (t_1+\cdot )
-\phi_\varepsilon (t_2+\cdot )$ for suitable $\psi\in\smooth (\real
).$ This yields the claim.
\end{proof}

Next we verify that the $L^2$-norm of a random variable of the form
\refeq{eq2.10} equals the norm of $\phi$ in a corresponding
homogeneous Sobolev space. For later purposes we first state an auxiliary result
that is valid in all dimensions.

\begin{lemma}\label{le:2.15} 
Assume that $H\in (0,1)$ and the
functions $\phi , \psi \in \smooth (\real^n)$ satisfy
$\int_{\real^n}\phi\, dx = \int_{\real^n}\psi\, dx =0.$ Then
\beqla{eq:2.104} 
&&\int_{\real^n}\int_{\real^n} \frac{1}{2}\big(
|u|^{2H}+|v|^{2H}-|u-v|^{2H}\big) \phi (u)\overline\psi (v)\, dudv\\ &=&
-2^{n+2H-1}\pi^{n/2}\frac{\Gamma (n/2+H)}{\Gamma (-H)})(\phi
,\psi)_{\homog^{-n/2-H,2}(\real^n)}.\nonumber 
\eeq
\end{lemma}.

\begin{proof}
We first claim that for $\alpha \in (0,n)$ 
\beqla{eq:2.105}
&&\int_{\real^n}\int_{\real^n} |u-v|^{-\alpha} \phi (u)\overline{\psi
(v)}\, dudv = (2\pi )^{n/2}d_{n,\alpha}\int_{\real^n}|\xi
|^{\alpha-n}\widehat\phi (\xi )\overline{\widehat\phi (\xi )}\, d\xi .
\eeq 
This is immediate by \refeq{eq:1.90} and the Parseval formula
since the left hand side above can be written as
$\int_{\real^n}g\overline{\psi}\, dx$ where $g$ is obtained as the
convolution $g=u_\alpha*\phi$, whence its Fourier transform equals
$(2\pi )^{n/2}d_{n,\alpha}|\xi|^{n-\alpha}\widehat \phi (\xi )$.  By
the assumption we see that the Fourier transforms of $\phi$ and $\psi$
satisfy $|\phi (\xi )|,|\psi (\xi )|\leq c|\xi |$ near the
origin. Moreover, they decay polynomially as $|\xi |\to\infty .$ These
observations verify that the right hand side of \refeq{eq:2.105} is
analytic as a function of $\alpha$ in a neighbourhood of the open line
segment $\alpha\in (-2,n).$ Since the left hand side of
\refeq{eq:2.105} is likewise analytic in the same neighbourhood we
deduce by analytic continuation that \refeq{eq:2.105} holds true for
all $\alpha\in (-2,n).$ The claim follows as we substitute $\alpha
=-2H$ in \refeq{eq:2.105} and observe that by Fubini the terms
$|u|^{2H}$ and $|v|^{2H}$ make no contribution to the integral in the
left hand side of \refeq{eq:2.104}.
\end{proof}

\begin{corollary}\label{co:2.1} Let $H\in (0,1)$ and assume that
 $\phi_1,\phi_2\in\smooth (\real)$ are real-valued. Then
$$
\expec \left( (\int_\real \phi_1'(t)X_t\, dt)(\int_\real
\phi_2'(t)X_t\, dt)\right ) =a_H(\phi_1,\phi_2)_{\homog^{\half -H,2}},
$$
where $a_H:=\sin (\pi H )\Gamma (1+2H)>0.$ Especially, there is an isometric and bijective
isomorphism $J: E_{(-\infty,\infty )}\to \homog^{1/2-H,2}(\real )$
so that for each interval $(t,t')\subset\real$ we have  $J(E_{(t,t')})=\homog_0^{1/2-H,2}((t,t')).$
\end{corollary}

\begin{proof} 
Let us denote
$$
  A:=\expec \left( (\int_\real \phi_1'(t)X_t\, dt)(\int_\real
  \phi_2'(t)X_t\, dt)\right )
$$
By the definition of the fractional Brownian motion with the Hurst
parameter $H\in (0,1)$ we have 
\beqla{eq2.11}
A&=&\half\int_{\real\times
\real}\phi_1'(u)_2\phi'(s)(|s|^{2H}+|u|^{2H}-|s-u|^{2H})\,ds\, du \\
&=&a_H(\phi'_1,\phi'_2)_{\homog^{-1/2-H,2}}
=a_H(\phi_1,\phi_2)_{\homog^{\half -H,2}}.  
\eeq 
Above we used Lemma \ref{le:2.15} to obtain the first
equality. Observe that the functions $\phi_1'$ and $\phi_2'$
automatically have mean zero.  The last equality follows directly from
the fact that the Fourier transfrom of $\phi_j'$ equals $i\xi
\widehat\phi_j (\xi )$, $j=1,2.$ The constant is simplified by
applying the standard formulas for the Gamma functions, see
e.g. \cite[5.2.4]{Ahlfors}. The last statement of the Corollary
follows immediately by Lemma \ref{le:2.1}.
\end{proof}

\begin{remark}\label{re:2.1}
Note that $a_H$ takes the value $1$ for $H=1/2$ and tends to zero
as $H\to 1^-$ or $H\to 0^+.$
\end{remark}

Let us observe that if the supports of $\phi_1$ and $\phi_2$ are
disjoint, we are free to integrate by parts in \refeq{eq2.11} and
obtain the formula 
\beqla{eq:2.15} &&\expec \left( (\int_\real
\phi_1'(t)X_t\, dt)(\int_\real \phi_2'(t)X_t\, dt)\right )\\
&=&H(2H-1)\int_{\real\times
\real}\frac{\phi_1(u)\phi_2(v)}{|u-v|^{2-2H}}
\, dudv. \nonumber 
\eeq 
Here it is interesting to observe
the sign of the factor $H(2H-1)$ for different values of the Hurst
parameter $H.$

We are now ready to prove the main result of the paper.

\begin{theorem}\label{th:2.1}
Fractional Brownian motions with $H\in(0,1)$ possess local
independence. Moreover, there is a constant $r_H\ge0$ {\rm (}with
$r_H>0$ for $H\not=1/2${\rm )} such that 
\beqla{eq:2.19}
\cos(\sphericalangle (E_{t_1,\varepsilon},E_{t_2,\varepsilon})&=&
r_H(\varepsilon/|t_1-t_2|)^{2-2H}+ O(\varepsilon^{3-2H})\quad
\mbox{as}\;\; \varepsilon\to 0,  \nonumber 
\eeq 
and (with some $\delta_H>0$)
\beqla{eq:2.190} 
I (E_{t_1,\varepsilon}:E_{t_2,\varepsilon})
&=&\frac{1}{2}r^ 2_H(\varepsilon/|t_1-t_2|)^{4-4H} +
O(\varepsilon^{4-4H+\delta_H})\quad \mbox{as}\;\; \varepsilon\to
0. \nonumber 
\eeq
\end{theorem}

\begin{proof} 
By scaling invariance and stationarity it is equivalent to show
that 
\beqla{eq:2.20} \cos (\sphericalangle (E_{(0,1)},E_{(k,k+1)}))&=&
r_Hk^{2H-2}+ O(k^{2H-3}) \quad\mbox{as}\;\; k\to\infty \quad \mbox{and}  \\
 I (E_{(0,1)}:E_{(k,k+1)})&=&
\frac{1}{2}r_H^2k^{4H-4}+ O(k^{4H-5}) \quad\mbox{as}\;\; k\to\infty.
\eeq 
Denote $s:=1/2-H\in (-1/2,1/2)$ together with  $A:=\homog^{s,2}_0(k,k+1)$
and $B:=\homog^{s,2}_0(0,1)$,
considered as subspaces of the Hilbert space $\homog^{s,2}(\real ).$
Let $P_A$ (resp. $P_B$) stand for the orthogonal projection on $A$ (resp. $B$).
We will consider the operator
$$
S:= P_B:A\to B.
$$
Since $S=(P_BP_A)_{|A}$ and $(P_BP_A)_{|A^\perp}=0$, we obtain that
$\| S\|=\| P_BP_A\|$ and $\| S\|_{HS}=\| P_BP_A\|_{HS}$. Hence
Corollaries \ref{co:4.1} and \ref{co:2.1} yield that
\beqla{eq:2.21} \cos (\sphericalangle (E_{(0,1)},E_{(k,k+1)}))=\|S\|\quad 
\eeq
and
\beqla{eq:2.211} 
 \frac{1}{2}\| S\|_{HS}^2\leq I (E_{(0,1)},E_{(k,k+1)})
\leq  \frac{1}{2}\| S\|_{HS}^2(1+\| S\|)
\eeq
as soon as $\| S\|<1/2$. 

In order to estimate the norm and the Hilbert-Schmidt norm of the
operator $S$ we will make use of the decay of the kernel in
\refeq{eq:2.15}, and the even faster decay of its derivatives. For
that end we need to first factorize $S$ properly through a suitable
 integral operator. Assume thus that $k\geq 2$ and  $\phi\in\smooth
(k,k+1)\subset A.$ Then by definitions and formula \refeq{eq:2.15} we
see that $S\phi\in B$ is the unique element that satisfies for each
$\psi\in \smooth (0,1)$
\beqla{eq:2.22}
(S\phi,\psi)_{\homog_0^{s,2}(0,1)}
&=&(\phi,\psi)_{\homog_0^{s,2}(\real )}
=H(2H-1)\int_{(0,1)\times (k,k+1)}\frac{\phi (y)\psi (x)}{|x-y|^{2-2H}}\, dxdy
\nonumber\\ 
&=&\int_0^1\psi (x)(R\phi)(x)\,
dx, 
\eeq 
where $R$ 
stands for the integral operator
$$
R\phi (x):= H(2H-1)\int_{(k,k+1)}\frac{\phi (y)}{|x-y|^{2-2H}}\, dy.
$$
By the smoothness of the kernel we immediately see that $R$ is well-defined and, in fact
$$
R\, (\homog_0^{s,2}(k,k+1))\subset W^{1,2}((0,1)).
$$
Let $G:\homog^{-s,2}(0,1)\to \homog_0^{s,2}((0,1))$ be the isometric
isomorphism from \refeq{eq:1.00}. 
According to \refeq{eq:2.22} we may  factorize
$$
S=GR.
$$

Let $V: \homog_0^{s,2}(k,k+1)\to\homog^{-s,2}(0,1)$ 
stand for the one-dimensional operator
$$
V\phi (x):= \int_{(k,k+1)}{\phi (y)}\, dy,\quad \mbox{for} \:x\in (0,1).
$$
Thus $V\phi$ is constant on $(0,1).$ We  decompose
\beqla{eq:2.24}
S &=& H(2H-1)k^{2H-2}GV +G\big( R-H(2H-1)k^{2H-2}V \big).\nonumber
\eeq
If we show that
\beqla{eq:2.25}
&&\|\Big( R-H(2H-1)k^{2H-2}V\Big):
  \homog_0^{s,2}(k,k+1)\to\homog^{-s,2}(k,k+1)\|_{HS}\\
&=& O(k^{2H-3}),\nonumber
\eeq
then, according to \refeq{eq:2.21}-\refeq{eq:2.211} and the fact that
for the one-dimensional operator $GV$ it holds that $\| GV \|_{HS}=\|
GV\|$ (the value is independent of $k$), both of the asymptotics in
\refeq{eq:2.20} follow immediately. Here we also keep in mind that the
Hilbert-Schmidt norm always dominates the operator norm.

Observe towards \refeq{eq:2.25} that for $x\in (0,1)$ and $\phi\in \smooth (k,k+1)$
we may write
$$
\big( (R-(H(2H-1)k^{2H-2}V)\phi \big)(x)=c\int_{(k,k+1)}u(x,y)\phi (y)\, dy,
$$
where a simple computation shows that the the kernel
$u(x,y)=|x-y|^{2H-2}-k^{2H-2}$ satisfies
$$
\|\big((\frac{d}{dx})^\alpha(\frac{d}{dy})^\beta u\big)(x,\cdot)
  \|_{L^\infty (k,k+1)}
\leq ck^{2H-3},\quad \alpha,\beta\in \{ 0,1\},\ x\in(0,1).
$$
By Lemma \ref{le:1.10}(iii) we have 
$\|\cdot \|_{\homog^{-s,2}((k,k+1))}\leq \|\cdot \|_{W^{1,2}((k,k+1))}$.
Hence the previous estimates yield for  fixed $x\in (0,1)$ the estimate
\beqla{eq:2.28}
\| u(x,\cdot )\|_{\homog^{-s,2}((k,k+1))}
\leq\| u(x,\cdot )\|_{W^{1,2}((k,k+1))}\leq c'k^{2H-3}
\eeq
and, similarly
\beqla{eq:2.29}
\|(\frac{d}{dx}) u(x,\cdot )\|_{\homog^{-s,2}((k,k+1))}
\leq\| (\frac{d}{dx}) u(x,\cdot )\|_{W^{1,2}((k,k+1))}\leq c'k^{2H-3}.
\eeq

Assume that $\|\phi \|_{\homog_0^{s,2}((k,k+1))}=1$. 
The duality \refeq{eq:1.0}, estimates \refeq{eq:2.28} and
\refeq{eq:2.29} show that
$$
\max_{\alpha\in\{ 0,1\}} \|(\frac{d}{dx})^\alpha
\Big( \big(R-(H(2H-1)k^{2H-2}V\big)\phi\Big)\|_{L^\infty (0,1)}
\leq  c'k^{2H-3}.
$$
This especially implies that
\beqla{eq:2.30}
\|\Big( R-(H(2H-1)k^{2H-2}V\Big) :\homog_0^{s,2}(k,k+1)\to W^{1,2}(k,k+1)\|
\leq  c_2k^{2H-3}.\nonumber
\eeq
Let us denote by $I:W^{1,2}((0,1))\to \dot W^{-s,2}((0,1))$ the
natural imbedding.  According to Lemma \ref{le:1.10} (iv) we have 
$\|I\|_{HS}<\infty $. We finally obtain
\beqla{eq:2.31}
&&\|\big( R-(H(2H-1)k^{2H-2}V\big) :
  \homog_0^{s,2}(k,k+1)\to \homog^{-s,2}(k,k+1)\|_{HS}\nonumber\\
&\leq& \| I\|_{HS}\|\big( R-(H(2H-1)k^{2H-2}V\big) :
  \homog_0^{s,2}(k,k+1)\to  W^{1,2}(k,k+1)\|\nonumber\\
&\leq & c_3k^{2H-3}.\nonumber
\eeq
This establishes \refeq{eq:2.25} and completes the proof of the theorem.
\end{proof}

\begin{remark}\label{rem:2.2} 
A closer inspection of the above proof reveals that the constant $r_H$
in Theorem \ref{th:2.1} satisfies 
$r_H=H|2H-1|\|\chi_{(0,1)}\|^2_{\homog^{H-1/2,2}(0,1)}$. 
Especially, $r_H$ tends to zero as $H\to 1/2$. Moreover, one also
checks that it is possible to choose $\delta_H=\min(1,2-2H)$.
\end{remark}

After Theorem \ref{th:2.1} it is natural to ask whether similar
phenomena take place if only one of the intervals in consideration
tends to a point. The answer is positive again. Heuristically one
might expect that the speed of convergence is only half of what it was
before, and this actually turns out to be true.

\begin{theorem}\label{th:2.2}
Let $t>0$. Then there are constants $r'_H\geq 0$ {\rm (}with $r'_H>0$ for
$H\not=1/2${\rm )} and $\delta'_H>0$ such that as $\varepsilon\to 0$ one has
\beqla{eq:2.191} 
\cos(\sphericalangle (E_{(-\infty,0)},E_{t,\varepsilon}))&=&
r'_H(\varepsilon/t)^{1-H}+ O(\varepsilon^{2-H})\quad
\quad\mbox{and} \nonumber\\
I (E_{(-\infty,0 )}:E_{t,\varepsilon})&=&
\frac{1}{2}(r'_H)^ 2(\varepsilon/t)^{2-2H}+ O(\varepsilon^{2-2H+\delta'_H})\quad
\mbox{as}\;\; \varepsilon\to 0.  \nonumber
\eeq
\end{theorem}

\begin{proof}
As in the proof of Theorem \ref{th:2.1} we apply scaling, Corollaries \ref{co:4.1} and
\ref{co:2.1}, and Lemma \ref{le:1.10}(iv) to the effect that it is equivalent to verify 
 in the limit $k\to \infty$ that we have
\beqla{eq:2.40} 
\|\dot S\| =r'_Hk^{H-1}+O(k^{H-2})\quad\mbox{and}\quad \|\widetilde S\|_{HS} =r'_Hk^{H-1}+O(k^{H-2}).
\eeq
Here $ \dot S=\dot G\dot R$, where $\dot G$ stands for the natural isomorphism
$\dot G:\homog^{-s,2}(k,k+1)\to \homog_0^{s,2}((k,k+1))$ provided by \refeq{eq:1.00}, 
 $s:=1/2-H,$ and
$$
\dot R:  \homog_0^{s,2}((-\infty,0))\to\homog^{-s,2}(k,k+1) 
$$ 
is the integral operator
$$
\dot R\phi (x):= H(2H-1)\int_{-\infty}^0\frac{\phi (y)}{|x-y|^{2-2H}}\, dy, \quad \mbox{for}\;  x\in (k,k+1).
$$

This time we consider the auxiliary  operator 
$\dot V : \homog_0^{s,2}((-\infty ,0))\to\homog^{-s,2}(k,k+1)$, where
$$
\dot V\phi (x):= H(2H-1)\int_{-\infty}^0\frac{\phi (y)}{|k-y|^{2-2H}}\, dy, \quad \mbox{for}\;  x\in (k,k+1).
$$
Thus $\dot V$ is one-dimensional since its image contains only constant functions.

According to Lemma \ref{le:1.2} it holds that
$$
\|\;  |k-\cdot|^{2H-2}\|_{\homog^{-s,2}((-\infty ,0))}=ck^{H-1}.
$$
Hence, by one-dimensionality and the duality \refeq{eq:1.0} we infer that
$$
\| \dot V:\homog_0^{s,2}((-\infty,0))\to\homog^{-s,2}(k,k+1)\|=c'k^{H-1}.
$$
By using again the decomposition $\dot S=\dot G\dot V+ \dot G(\dot R-\dot V)$
we deduce, as in the proof of Theorem \ref{th:2.1}, that the one-dimensionality of $\dot V$
and the Hilbert-Schmidt property of the natural imbedding $W^{1,2}((k,k+1))\to \homog^{-s,2}((k,k+1))$
(where the Hilbert-Schmidt norm is independent of $k$) enable us to deduce \refeq{eq:2.40}
as soon as we establish that
\beqla{eq:2.45}
\|\big( \dot R-\dot V\big) :\homog_0^{s,2}(-\infty ,0)\to W^{1,2}(k,k+1)\|
\leq  c_2k^{H-2}.
\eeq

Observe  that $\dot V-\dot R$ has the integral kernel
$\dot u(x,y):=2(2H-1)\big( |x-y|^{2H-2}-|k-y|^{2H-2}\big).$ Clearly \refeq{eq:2.45} follows from
duality and the estimate
\beqla{eq:2.46}
\sup_{x\in (k,k+1)}\|(\frac{d}{dx})^\alpha\dot u(x,\cdot )\|_{\homog^{-s,2}((-\infty ,0))}=O(k^{H-2})\quad \mbox{for}\; \alpha\in
\{ 0,1\} .
\eeq
In turn, for $\alpha =1$ this estimate is a direct consequence of Lemma \ref{le:2.1}. In order to verify
it for $\alpha =0$, we fix $x\in (k,k+1)$ and  apply the same Lemma  as follows:
\beqla{eq:2.47}
\|\dot u(x,\cdot )\|_{\homog^{-s,2}((-\infty ,0))}
&=&c\|\int_k^ x |t-\cdot|^{2H-3}\,dt\|_{\homog^{-s,2}((-\infty ,0))}\nonumber\\
&\leq& c\int_k^ x \|\; |t-\cdot|^{2H-3}\|_{\homog^{-s,2}((-\infty ,0))}dt\leq c'k^{H-2}.\nonumber
\eeq
In the second inequality above we made use of the Minkowski inequality for Banach space norms.
\end{proof}

The remaining cases are simpler to handle and they are collected in the
following theorem.

\begin{theorem}\label{co:2.2}
\noindent {\rm (i)}\quad Let $H\not=\half$. Then
$I(E_{(-\varepsilon,0)}:E_{(0,\varepsilon)})=\infty$ for any
$\varepsilon>0$.

\noindent {\rm (ii)}\quad Let $H\not=\half$. Then
$I(E_{(-\infty,-\varepsilon)}:E_{(\varepsilon,\infty)})=\infty$ for any
$\varepsilon>0$.

\noindent {\rm (iii)}\quad $\sphericalangle(E_{(-\infty,0)},E_{(0,\infty)})>0$.

\noindent {\rm (iv)}\quad Let $t_1<t<t_2$ be arbitrary. Then for
small enough $\varepsilon>0$ it holds that
\beqla{eq:2.50}
I(E_{(-\infty,t_1)\cup (t_2,\infty)}: 
E_{t,\varepsilon})\leq c\varepsilon^{H-1}.
\eeq
\end{theorem}

\begin{proof}

(i) Assume the contrary, that is, $I(E_{(-\varepsilon,0)}:E_{(0,\varepsilon
)})<\infty$ for some $\varepsilon>0$. Since FBM possesses local
independence, its infinitesimal space is trivial, that is, 
$
\bigcap_{n=1}^\infty E_{(0,\pm\varepsilon/n)}=\{0\}
$
(otherwise the Gaussian space would have uncountable dimension; see 
Proposition 5 of \cite{Mandjes}). By Theorem 1 of \cite{Tutubalin},
this implies the corresponding relation for $\sigma$-algebras, i.e.\
$
\bigcap_{n=1}^\infty\sigma(E_{(0,\pm\varepsilon/n)})=\{\Omega,\emptyset\}
$
up to sets of measure 0 or 1. Theorem \ref{shannonbasic} (iv) then yields
that $\lim_{n\to\infty}I(E_{(-\varepsilon/n,0)}:E_{(0,\varepsilon/n
)})=I(\{\Omega,\emptyset\}:\{\Omega,\emptyset\})=0$. On the other
hand, we have $I(E_{(-\varepsilon,0)}:E_{(0,\varepsilon
)})>0$ when $H\not=\half$. Now, however, the self-similarity of FBM
implies that $I(E_{(-\varepsilon/n,0)}:E_{(0,\varepsilon/n
)})$ does not depend on $n$, and we get a contradiction.

(ii) By self-similarity, Theorem \ref{shannonbasic} (iii) and the
previous claim, we have
\begin{eqnarray*}
I(E_{(-\infty,-\varepsilon)}:E_{(\varepsilon,\infty)})
&=&\lim_{n\to\infty}I(E_{(-\infty,-\varepsilon/n)}:E_{(\varepsilon/n,\infty)})\\
&=&I(E_{(-\infty,0)}:E_{(0,\infty)})\\
&\ge& I(E_{(-\varepsilon,0)}:E_{(0,\varepsilon)})=\infty.
\end{eqnarray*}

(iii) This is an immediate consequence of Lemma \ref{le:1.10}(i) and Lemma
\ref{le:2.1}, since together they imply that for a dense set of
elements $X_1\in E_{(-\infty,0)}$ and $X_2\in E_{(0,\infty)}$ we have that
$$
\max (\| X_1\|, \| X_2\| )\leq c\| X_1-X_2\|.
$$

(iv) Write $A_1=E_{(-\infty ,t_1)}$, $A_2=E_{(t_2,-\infty )}$, and
$B=E_{t,\varepsilon}.$ Since the angle between the
subspaces $A_1$ and $A_2$ is positive, we see that $A:=\cspan
(A_1\bigcup A_2)$ is naturally isomorphic (not necessarily isometric)
to the direct sum $(A_1{\oplus} A_2)_{\ell^2}.$ In this isomorphism
the operator $P_B:A\to B$ conjugates to the operator
$[P_B:A_1\to B,\ P_B:A_2\to B]$, whose Hilbert-Schmidt norm is bounded
by $c\varepsilon^{1-H}$ by Theorem \ref{th:2.2}. This proves the
claim.
\end{proof}

\section{Generalizations and open questions}\label{se:open}

The most natural generalization of FBM to $\real^n$ is the Levy FBM,
which is defined as the Gaussian process $X_u$ indexed by the
parameter $u\in\real^n$ and having the covariance structure
$$
\expec X_uX_v= \frac{1}{2}\big( |u|^{2H}+|v|^{2H}-|u-v|^{2H}\big).
$$
Here $H\in (0,1)$.  As in the one-dimensional case this process has a
version that has Hölder continuous realizations. We refer to
\cite[Chapter 18]{Kahane} for the existence and basic properties of
$n$-dimensional Levy FBM.  We will sketch the proof of an
$n$-dimensional version of Theorem \ref{th:2.1}.  For that end we
first present  an auxiliary result.

\begin{lemma}\label{le:3.1}
Let $n\geq 2$ and $s\in (-n/2-1,-n/2).$ Then there is a constant $c>0$
such that for every $\phi\in \smooth (B(0,1))$ with $\int_{\real^n}\phi \, dx =0$
and $f\in C^{n+1}(\overline{B(0,1)})$ 
it holds that
$$
|\int_{B(0,1)}f\phi\, dx|\leq c\|\phi \|_{\homog_0^{s,2}(B(0,1))}\sum_{1\leq |\alpha |\leq n+1}\| D^\alpha f\|_{L^\infty (\overline{B(0,1)})}.
$$
\end{lemma}

\begin{proof}
Observe that in the left hand side we may replace $f$ by $f-m,$ where
$m$ is the average of $f$ over the ball $B(0,1)$. Hence we may assume
that $\| f\|_{L^\infty (\overline{B(0,1)})}$ is dominated by $\|
Df\|_{L^\infty (\overline{B(0,1)})}$.  It follows that it is enough to
prove the stated estimate where one sums over all $|\alpha |\leq n+1$
in the right hand side. But it is easy to extend $f$ to an element
$\widetilde f \in W^{n+1,2}(\real^n)$ with norm less than constant
times $\sum_{|\alpha |\leq n+1}\| D^\alpha f\|_{L^\infty
(\overline{B(0,1)})}$.  The claim follows now by duality since formally
$W^{n+1,2}(B(0,1))\subset
\homog^{-s,2}(B(0,1))=\homog_0^{s,2}(B(0,1))'. $
\end{proof}

\begin{theorem}
Let $\{ X_s\}_{s\in\real^n}$ be an $n$-dimensional Levy FBM with Hurst
parameter $H\in (0,1).$ For any ball $B\subset\real^n$ let $E_B$ be
the $L^2$-space generated by the differences $\{
X_{s_1}-X_{s_2}\joille s_1,s_2\in B\} .$ Then, if $s_1\not=s_2$ the
subspaces $E_{B(s_1,\varepsilon )}$ and $E_{B(s_2,\varepsilon )}$ are
asymptotically independent as $\varepsilon\to 0.$ Moreover, there are
positive constants $c_1,c_2>0$ such that
$$
c_1\varepsilon^{2H-2}\leq \cos(\sphericalangle (E_{B(s_1,\varepsilon
)},E_{B(s_2,\varepsilon )}))\leq c_2\varepsilon^{2H-2}.
$$
\end{theorem}

\begin{proof}
The proof is analogous to the proof of Theorem \ref{th:2.1}.  First
of all, the lower bound is an immediate consequence of the
one-dimensional case since the restriction of the process to a line
through the points $s_1,s_2$ is a one-dimensional FBM. In order to
deduce the upper bound we observe that according to Lemma
\ref{le:2.15} and an easy analogue of Lemma \ref{le:2.1} the cosine of
the angle between the spaces is given by the quantity
$$
A:=-\frac{1}{2}\sup_{\phi,\psi}\int_{\real^n}\int_{\real^n}
|u-s|^{2H}\phi (u)\overline{\psi (s)}\, duds ,
$$ 
where the supremum is taken over all functions  $\phi\in \smooth (B(0,1))\cap W_0^{-n/2-H,2}(B(0,1))$ and
$\psi\in \smooth(B(ke_1,1))\cap W_0^{-n/2-H,2}(B(ke_1,1))$,  with unit norm and zero mean. Here $k=|s_1-s_2|/\varepsilon >0$. Observe that
we used the obvious scaling and rotation invariance of the Levy
FBM. By a twofold application of Lemma \ref{le:3.1} it follows that
$$
A\lesssim \sup_{u\in B(0,1), s\in B(ke_1,1)}\sum_{1\leq |\alpha |\leq
n+1, 1\leq |\beta |\leq n+1}\big| D^\alpha_uD^\beta_s \big(
|u-s|^{2H}\big)\big| \sim k^{2H-2}.
$$
\end{proof}

Our results raise several interesting open problems related to local
independence of stochastic processes. We expect that the methods of
the present paper are pretty much restricted to dealing with the
FBM, although they may help in obtaining insights and conjectures
regarding the following questions. 

\medskip

{\bf Q.1}\quad Let $X=\{ X_t\}_{t\in\real}$ be a Gaussian process with
continuous paths and stationary increments. Find necessary and
sufficient conditions for the local independence property, e.g.\ in
terms of the spectral measure of $X$, or in terms of the variance
function $v(t)=\expec X_t^2$.

\medskip

With regards to Question 1, we can note a couple of obvious obstacles
for local independence. First, if the process is $L^2$-differentiable,
the value of the derivative process belongs to the infinitesimal
sigma-algebra around a point (see \cite{Tutubalin}), and apart from
trivial cases this will destroy local independence. Second, periodic
processes, like the periodic Brownian bridge defined by the variance
function
$$
v(t)=\expec X_t^2=(t \mod 1)(1-(t \mod 1)),
$$
clearly do not satisfy local independence for all times. Periodic
components are reflected as atoms of the spectral measure. But are
non-smoothness and continuity of spectrum already sufficient for local
independence? 

One can also ask for a local characterization:

\medskip

{\bf Q.2}\quad Let $(X_t)$ again be a Gaussian process with stationary
increments. Give conditions on the variance function $v(t)$ in a
neighbourhood of the origin and in a neighbourhood of the point
$|t_1-t_2|$ that would guarantee local independence with respect to
points $t_1,t_2$.

\medskip

{\bf Q.3}\quad Superposing Brownian bridges with different periods,
one can probably build examples of non-smooth processes where local
independence breaks over any rational distance. But is it possible to
construct a continuous but non-differentiable Gaussian process with
stationary increments that does not possess local independence over
any distance?

\medskip

{\bf Q.4} So far we have only focused on Gaussian processes. Our
information-based definition of local independence is, however,
meaningful for any kind of stochastic process. It is then interesting
to ask about the local independence of various dependent
processes. For example, do fractional L\'evy processes have this
property?

\bibliographystyle{amsalpha}

\end{document}